\documentclass[12pt,a4paper]{amsart}
\usepackage{t1enc}
\usepackage[latin1]{inputenc}
\usepackage{amssymb}
\usepackage{amsthm}
\usepackage{amsmath}
\usepackage{latexsym}
\usepackage{array}
\usepackage{mathrsfs}
\usepackage[all]{xy}
\usepackage{graphics}
\usepackage{xcolor}
\usepackage{stmaryrd}
\usepackage{color}
\usepackage{enumerate}
\title[Rationality of cycles on function field of exceptional varieties]{Rationality of cycles on function field of exceptional projective homogeneous varieties}
\date{2 June 2013}
\author{{Raphael Fino}}
\address
{UPMC Sorbonne Universit\'es\\
Institut de Math\'ematiques de Jussieu\\
Paris\\
FRANCE}
\address
{{\it Web page:}
{\tt www.math.jussieu.fr/\~{ }fino}}
\email {fino {\it at} math.jussieu.fr}
\numberwithin{equation}{section}
\theoremstyle{definition}

\newtheorem{rem}[equation]{Remark}
\newtheorem{lemme}[equation]{Lemma}

\newtheorem{prop}[equation]{Proposition}
\newtheorem{thm}[equation]{Theorem}
\newtheorem{cor}[equation]{Corollary}
\newcommand{\twoheadlongrightarrow}{\relbar\joinrel\twoheadrightarrow}
\begin{document}

\begin{abstract}
In this article we prove a result comparing rationality of algebraic cycles over the function field
of a projective homogeneous variety under a linear algebraic group of type $F_4$ or $E_8$
and over the base field, which can be of any characteristic.

\end{abstract}

\maketitle

\section{Introduction}

Let $G$ be a linear algebraic group of type $F_4$ or $E_8$ over a field $F$
and let $X$ be a projective homogeneous $G$-variety.
We write $Ch$ for the Chow group with coefficient in $\mathbb{Z}/p\mathbb{Z}$,
with $p=3$ when $G$ is of type $F_4$ and $p=5$ when $G$ is of type $E_8$.
The purpose of this note is to prove the following theorem dealing with 
rationality of algebraic cycles on function field of such a projective homogeneous $G$-variety.

\begin{thm}
\textit{For any equidimensional variety $Y$, the change of field homomorphism
\[Ch(Y)\rightarrow Ch(Y_{F(X)})\]
is surjective in codimension $<p+1$.
It is also surjective in codimension $p+1$ for a given $Y$ provided that 
$1\notin \text{deg}\;Ch_0(X_{F(\zeta)})$ for each generic point
$\zeta \in Y$.}
\end{thm}

The proof is given in section 3. 

In previous papers (\cite{ARC1}, \cite{ARC2}, after the so-called Main Tool Lemma by A.\,Vishik, cf \cite{GPQ}, \cite{RIC}), similar issues about rationality of cycles,
with quadrics instead of exceptional projective homogeneous varieties, have been treated.
The above statement is to put in relation with \cite[Theorem 4.3]{OSNV}, where \emph{generic splitting
varieties} have been considered. Also, Theorem 1.1 is contained in  \cite[Theorem 4.3]{OSNV}
if $\text{char}(F)=0$.

On the one hand, our method of proof is basically the method used to prove 
\cite[Theorem 4.3]{OSNV}. On the other hand, our method mainly relies
on a motivic decomposition result for projective homogeneous varieties due to 
V.\,Petrov, N.\,Semenov and K.\,Zainoulline (cf \cite[Theorem 5.17]{J-inv}).
It also relies on a linkage between the $\gamma$-filtration and Chow groups, 
in the spirit of \cite{Gamma}.
Our method works in any characteristic and is particularly suitable
for groups of type $F_4$ and $E_8$ mainly because the latter have an opportune $J$-invariant.

\medskip

In the aftermath of Theorem 1.1, we get the following statement dealing with
integral Chow groups (see \cite[Theorem 4.5]{OSNV}).

\begin{cor}
\textit{If $p\in \text{deg}\;CH_0(X)$ then for any equidimensional variety $Y$, the change of field homomorphism
\[CH(Y)\rightarrow CH(Y_{F(X)})\]
is surjective in codimension $<p+1$.
It is also surjective in codimension $p+1$ for a given $Y$ provided that 
$1\notin \text{deg}\;Ch_0(X_{F(\zeta)})$ for each generic point
$\zeta \in Y$.}
\end{cor}

\begin{rem}
Our method of proof for Theorem 1.1 works for groups of type $G_2$ as well (with p=2).
However, the case of $G_2$ can be treated in a more elementary way if $\text{char}(F)=0$.

Indeed, it is known that to each group $G$ of type $G_2$ one can associate
a $3$-fold Pfister quadratic form $\rho$ such that, by denoting $X_{\rho}$ the Pfister quadric associated with $\rho$, the variety $X$ has a rational point over $F(X_{\rho})$ and vice-versa. Thus, for any equidimensional variety $Y$, one has the commutative diagram
\[\xymatrix{
 Ch(Y)   \ar[r] \ar[d] &  Ch(Y_{F(X)}) \ar[d]\\
   Ch(Y_{F(X_{\rho})})\ar[r]      &  Ch(Y_{F(X_{\rho}\times X)}) 
    }\]
where the right and the bottom maps are isomorphisms. 
Furthermore, as suggested in \cite[Remark on Page 665]{RIC} (where the assumption $\text{char}(F)=0$ is required), the change of 
field homomorphism $Ch(Y)\rightarrow Ch(Y_{F(Q)})$
is surjective in codimension $<3$. 
\end{rem}

\smallskip
\noindent
{\sc Acknowledgements.}
I gracefully thank Nikita Karpenko for sharing his great knowledge
and his valuable advice.

\section{Filtrations on projective homogeneous varieties}
In this section, we prove two propositions which play a crucial role in the proof of Theorem 1.1.

\medskip

First of all, we recall that for any smooth projective variety $X$ over a field $E$, one can consider two particular filtrations on the Grothendieck ring $K(X)$ (see \cite[\S 1.A]{Gamma}), i.e the
$\gamma$-filtration and the topological filtration, whose respective terms of 
codimension $i$ are given by
\[\gamma^i(X)=\langle c_{n_1}(a_1)\cdots  c_{n_m}(a_m)\,|\, n_1+\cdots+n_m\geq i \:\text{and}\:a_1, \dots ,a_m \in K(X)\rangle\]
 \noindent and
\[\tau^i(X)=\langle [\mathcal{O}_Z]\,|\,Z\hookrightarrow X \: \text{and}\: \text{codim}(Z)\geq i \rangle,\]
\noindent where $c_n$ is the $n$-th Chern Class with values in $K(X)$ and $[\mathcal{O}_Z]$
is the class of the structure sheaf of a closed subvariety $Z$.
We write $\gamma^{i/i+1}(X)$ and $\tau^{i/i+1}(X)$ for the respective quotients.
For any $i$, one has $\gamma^i(X)\subset \tau^i(X)$ and one even has 
$\gamma^i(X)=\tau^i(X)$ for $i\leq 2$.
We denote by $pr$ the canonical surjection
\[\begin{array}{rcl}
CH^i(X) & \twoheadlongrightarrow & \tau^{i/i+1}(X) \\
\left[Z\right] & \longmapsto & [\mathcal{O}_Z]
\end{array},\]
where $CH$ stands for the integral Chow group. 

The method of proof of the following proposition is largely inspired by the proof of \cite[Theorem 6.4 (2)]{CPD}.

\begin{prop}
\textit{Let $G_0$ be a split semisimple linear algebraic group over a field $F$ and let $B$ be a Borel subgroup of $G_0$.
There exist an extension $E/F$ and a cocycle $\xi \in H^1(E,G_0)$ such that the topological filtration
and the $\gamma$-filtration coincide on $K({}_{\xi}(G_0/B))$.}
\end{prop}

\begin{proof}
Let $n$ be an integer such that $G_0\subset \textbf{GL}_n$ and let us set $S:=\textbf{GL}_n$ and $E:=F(S/G_0)$.
We denote by $\textbf{T}$ the $E$-variety $S\times_{S/G_0}\text{Spec}(E)$ given by the generic fiber of the projection
$S\rightarrow S/G_0$. Note that since $\textbf{T}$ is clearly a $G_0$-torsor over $E$, there exists a cocycle 
$\xi \in H^1(E,G_0)$ such that the smooth projective variety $X:=\textbf{T}/B_E$ is isomorphic to ${}_{\xi}(G_0/B)$.
We claim that the Chow ring $CH(X)$ is generated by Chern classes.

Indeed, the morphism $h:X\rightarrow S/B$ induced by the canonical $G_0$-equivariant morphism $\textbf{T}\rightarrow S$
being a localisation, the associated pull-back
\[h^{\ast}:CH(S/B)\longrightarrow CH(X)\]
is surjective. Furthermore, the ring $CH(S/B)$ itself is generated by Chern classes: by \cite[\S 6,7]{CPD} there exist a morphism
\begin{equation}\mathbb{S}(T^{\ast})\longrightarrow CH(S/B),\end{equation}
(where $\mathbb{S}(T^{\ast})$ is the symmetric algebra of the group of characters $T^{\ast}$
of a split maximal torus $T\subset B$) with its image generated by Chern classes. 
Moreover, the morphism (2.2) is surjective by \cite[Proposition 6.2]{CPD}.
Since $h^{\ast}$ is surjective and Chern classes commute with pull-backs, the claim is proved.

We show now that the two filtrations coincide on $K(X)$ by induction on dimension.
Let $i\geq 0$ and assume that $\tau^{i+1}(X)=\gamma^{i+1}(X)$. Since for any $j\geq 0$, one has 
$\gamma^{j}(X)\subset\tau^{j}(X)$, the induction hypothesis implies that 
\[\gamma^{i/i+1}(X)\subset \tau^{i/i+1}(X).\]
Thus, the ring $CH(X)$ being generated by Chern classes, one has $\gamma^{i/i+1}(X)=\tau^{i/i+1}(X)$
by \cite[Lemma 2.16]{cod2}. Therefore one has $\tau^i(X)=\gamma^i(X)$ and the proposition is proved. 
\end{proof}

Note that this result remains true when one consider a \textit{special} parabolique subgroup $P$ instead of $B$.

\medskip

\medskip

\medskip

Now, we prove a result which will be used in section 3 to get the second conclusion of Theorem 1.1.

We recall that for any smooth projective variety $X$ over a field and for any $i<p+1$,
the canonical surjection $pr:Ch^i(X)\twoheadrightarrow \tau^{i/i+1}(X)$ with $\mathbb{Z}/p\mathbb{Z}$-coefficient
is an isomorphism (cf \cite[\S 1.A]{Gamma} for example). The following proposition extends this fact to $i=p+1$ 
provided that $X$ is a projective homogeneous variety
under a linear algebraic group $G$ of type $F_4$ or $E_8$.

\begin{prop}
\textit{Let $X$ be a projective homogeneous variety
under a group $G$ of type $F_4$ or $E_8$, then the canonical surjection 
\[pr:Ch^{p+1}(X)\twoheadrightarrow \tau^{p+1/p+2}(X)\]
is injective.}
\end{prop}

\begin{proof}

The epimorphism $pr:Ch^{p+1}(X)\twoheadrightarrow \tau^{p+1/p+2}(X)$ coincides with the edge homomorphism of the spectral Brown-Gersten-Quillen
structure $E^{p+1,-p-1}_2(X)\Rightarrow K(X)$,
i.e $E^{p+1,-p-1}_r(X)$ stabilizes for $r>>0$ with 
$E^{p+1,-p-1}_{\infty}(X)=\tau^{p+1/p+2}(X)$,
 and for any $r\geq 2$ 
the differential $E^{p+1,-p-1}_r(X)\rightarrow E^{p+1+r,-p-r}_r(X)$
is zero, so that the epimorphism $pr$ coincides with the composition
\[Ch^{p+1}(X)\simeq E^{p+1,-p-1}_2(X)\twoheadrightarrow 
E^{p+1,-p-1}_3(X)\twoheadrightarrow \cdots \twoheadrightarrow
E^{p+1,-p-1}_{\infty}(X)=\tau^{p+1/p+2}(X).\]

\medskip

Now, it is equivalent in order to prove the proposition to prove that 
for any $r\geq 2$, the differential 
$E^{p+1-r,-p-2+r}_r(X)\rightarrow E^{p+1,-p-1}_r(X)$
is zero.

\medskip

First of all, since we work with $\mathbb{Z}/p\mathbb{Z}$-coefficient, by \cite[Theorem 3.6]{Adams}, 
the differential 
$E^{p+1-r,-p-2+r}_r(X)\rightarrow E^{p+1,-p-1}_r(X)$
is zero for any $r\geq 2$ with $r\neq p$. Hence, one only has to show that the 
differential $E^{1,-2}_p(X)\rightarrow E^{p+1,-p-1}_p(X)$ is zero.

\medskip

Let us consider the following composition given by the BGQ-structure
\[E^{1,-2}_{\infty}(X)\hookrightarrow \cdots \hookrightarrow 
E^{1,-2}_3(X)\hookrightarrow E^{1,-2}_2(X).\]
Note that one has $E^{1,-2}_{\infty}(X)\simeq E^{1,-2}_2(X)$
if and only if for any $r\geq 2$ the differential 
$E^{1,-2}_r(X)\rightarrow E^{1+r,-2-r+1}_r(X)$ is zero.
Therefore it is sufficient to prove that $E^{1,-2}_{\infty}(X)\simeq E^{1,-2}_2(X)$ to get that the 
differential $E^{1,-2}_p(X)\rightarrow E^{p+1,-p-1}_p(X)$ is zero.

\medskip
 
On the one hand, by the very defintion, the group $E^{1,-2}_{\infty}(X)$
is the first quotient $K^{(1/2)}_1(X)$ of the topological filtration on $K_1(X)$.
On the other hand, one has $E^{1,-2}_2(X)\simeq H^1(X,K_{2})$ 
(for any integers $p$ and $q$, one has $E^{p,q}_2(X)\simeq H^p(X,K_{-q})$).

Let us now consider the commutative diagram (cf \cite[\S 4]{Algeo})
\[\xymatrix{
      K^{(1/2)}_1(X) \ar@{^{(}->}[rr] && H^1(X,K_{2}) \\
       & H^0(X,K_{1})\otimes Ch^1(X) \ar[lu] \ar[ru]
    }\]
We claim that the natural map $ H^0(X,K_{1})\otimes Ch^1(X)\rightarrow H^1(X,K_{2})$ is an isomorphism.
Indeed
since $G$ is of type $F_4$ or $E_8$, it has only trivial Tits algebras, and
therefore, by \cite[Theorem]{Merk2}, one has
\[H^1(X,K_{2})\simeq H^1(X_{\text{sep}},K_{2})^{\Gamma},\]
where $\Gamma$ is the absolute Galois group of $F$. 
Moreover, since the variety $X_{\text{sep}}$ is cellular, by \cite[Proposition 1]{Merk2},
one has 
\[H^1(X_{\text{sep}},K_{2})\simeq
K_1F_{\text{sep}}\otimes Ch^1(X_{\text{sep}}).\]
Thus, since the Picard group of any homogeneous projective variety under a group of 
type $F_4$ or $E_8$ is rational (cf \cite[Example 4.1.1]{Gen}) and since $(K_1F_{\text{sep}})^{\Gamma}=K_1F=H^0(X,K_{1})$, one has
\[H^1(X,K_{2})\simeq K_1F\otimes Ch^1(X)\simeq H^0(X,K_{1})\otimes Ch^1(X),\]
and the claim is proved. Therefore, one has $E^{1,-2}_{\infty}(X)\simeq E^{1,-2}_2(X)$ and the proposition is proved.
\end{proof}

\begin{rem}
Assume that $G_0$ of \emph{strongly inner} type (e.g $F_4$ and $E_8$)
and consider an extension $E/F$ and a cocycle $\xi \in H^1(E, G_0)$.
By \cite[Theorem 2.2.(2)]{Pan1}, the change of field homomorphism
\[K({}_{\xi}(G_0/B)_E)\rightarrow K({}_{\xi}(G_0/B)_{\overline{E}})\simeq 
K(G_0/B)\]
is an isomorphism, where $\overline{E}$ denotes an algebraic closure of $E$.
Therefore, since the $\gamma$-filtration is defined in terms of Chern classes and the latter commute with pull-backs, the quotients of the $\gamma$-filtration on $K({}_{\xi}(G_0/B)_E)$
do not depend nor on the extension $E/F$ neither on the choice of $\xi \in H^1(E, G_0)$.
\end{rem}

\section{Proof of Theorem 1.1}

In this section, we prove Theorem 1.1. 

\medskip

First of all, note that the $F$-variety $X$ is $A$-trivial in the sense of \cite[Definition 2.3]{OSNV} 
(see \cite[Example 2.5]{OSNV}),
i.e for any extension $L/F$ with $X(L)\neq \emptyset$, the degree homomorphism $\text{deg}: Ch_0(X_L)\rightarrow \mathbb{Z}/p\mathbb{Z}$ 
is an isomorphism. Therefore, by \cite[Lemma 2.9]{OSNV}, the change of field homomorphism $Ch(Y)\rightarrow Ch(Y_{F(X)})$ is an isomorphism (in any codimension) 
if $1\in \text{deg}\,Ch_0(X)$.
Hence, one can assume that $1\notin \text{deg}\,Ch_0(X)$.

Now, we know from \cite[Table 4.13]{J-inv} that the $J$-invariant $J_p(G)$ of $G$ is equal to $(1)$ or $(0)$.
However, the assumption $J_p(G)=(0)$ implies that there exists a splitting field $K/F$ of degree coprime to $p$ 
(see \cite[Corollary 6.7]{J-inv}),
and in that case one has $Ch_0(X)\simeq Ch_0(X_K)$ and $1\in \text{deg}\,Ch_0(X_K)$ by $A$-triviality of $X$.
Thus, under the assumption $1\notin \text{deg}\,Ch_0(X)$, one necessarly has $J_p(G)=(1)$ and that is why we can
assume $J_p(G)=(1)$ in the sequel. 

\medskip

Since $X$ is $A$-trivial, one can use the following proposition (cf \cite[Proposition 2.8]{OSNV}).

\begin{prop}[Karpenko, Merkurjev]
\textit{Given an equidimensional $F$-variety Y and an integer $m$ such that for any $i$ and any point
$y\in Y$ of codimension $i$ the change of field homomorphism}
\[Ch^{m-i}(X)\rightarrow Ch^{m-i}(X_{F(y)})\]
\textit{is surjective, the change of field homomorphism}
\[Ch^{m}(Y)\rightarrow Ch^{m}(Y_{F(X)})\]
\textit{is also surjective.}
\end{prop}

Consequently, it is sufficient in order to prove the first conclusion of
Theorem 1.1 to show that for any extension $L/F$,
the change of field homomorphism
\begin{equation} Ch(X)\longrightarrow Ch(X_L) \end{equation}
is surjective in codimension $<p+1$.

\medskip

Moreover, the $F$-variety being generically split (see \cite[Example 3.6]{J-inv}), one can apply the
motivic decomposition result \cite[Theorem 5.17]{J-inv} to $X$ and get that the motive $\mathcal{M}(X,\mathbb{Z}/p\mathbb{Z})$
decomposes as a sum of twists of an indecomposable motive $\mathcal{R}_p(G)$ (in the same way as (3.5)). 
Note that the quantity and the value of those twists do not depend on the base field.
In particular, we get that for any extension $L/F$ and any integer $k$, the group $Ch^k(X_L)$ is isomorphic to a direct sum
of groups $Ch^{k-i}(\mathcal{R}_p(G)_L)$ with $0\leq i \leq k$.
Therefore, the surjectivity of (3.2) in codimension $<p+1$ is a consequence of the following proposition. 

\begin{prop}
\textit{For any extension $L/F$, the change of field}
\begin{equation} Ch(\mathcal{R}_p(G))\longrightarrow Ch(\mathcal{R}_p(G)_L) \end{equation}
\textit{is surjective in codimension $<p+1$.}
\end{prop}

\begin{proof}
Let $G_0$ be a split linear algebraic group of the same type of the type of $G$ and let $\xi \in H^1(F,G_0)$ be a cocycle
such that $G$ is isogenic to the twisted form ${}_{\xi}G_0$. We write $\mathfrak{B}$ for the Borel variety of $G$
(i.e $\mathfrak{B}={}_{\xi}(G_0/B)$, where $B$ is a Borel subgroup of $G_0$).

\medskip

By \cite[Theorem 5.17]{J-inv}, one has the motivic decomposition
\begin{equation}\mathcal{M}(\mathfrak{B},\mathbb{Z}/p\mathbb{Z})\simeq \bigoplus_{i\geq 0}\mathcal{R}_p(G)(i)^{\oplus a_i},
\end{equation}
where $\Sigma_{i\geq 0}a_i t^i=P(CH(\overline{\mathfrak{B}}),t)/P(CH(\overline{\mathcal{R}_p(G)}),t)$,
with $P(-,t)$ the \textit{Poincar\'{e} polynomial}. Thus, for any integer $k$,we get the following decomposition
concerning Chow groups
\begin{equation}
Ch^k(\mathfrak{B}_L)\simeq \bigoplus_{i\geq 0} Ch^{k-i}(\mathcal{R}_p(G)_L)^{\oplus a_i}
\end{equation}

First of all, the homomorphism (3.4) is clearly surjective in codimension 0 since one has
$Ch^0(\mathcal{R}_p(G)_L)=\mathbb{Z}/p\mathbb{Z}$ for any extension $L/F$. 
Then, $Ch^1(\overline{\mathfrak{B}})$ is identified with the Picard group $\text{Pic}(\overline{\mathfrak{B}})$
and is rational (see \cite[Example 4.1.1]{Gen}).
Furthermore, thanks to the Solomon Theorem for example (see \cite[\S 2.5]{Gen}), one can compute the coefficients $a_i$'s:
we get $a_0=1$ and $a_1=\text{rank}(G)=\text{rank}(Ch^1(\overline{\mathfrak{B}}))$. Thus, the isomorphism (3.6) implies that
$Ch^1(\mathcal{R}_p(G)_L)=0$ for any extension $L/F$.

\medskip

\medskip

\medskip

We have already shown that the homomorphism (3.4) is surjective in codimension $0$ and $1$.
The following lemma implies the surjectivity in codimension $2$ and $3$ (and therefore proves the first conclusion of
Theorem 1.1 if $G$ is of type $F_4$).

\begin{lemme}
\textit{Under the assumption $J_p(G)=(1)$, one has}
\[Ch^2(\mathcal{R}_p(G))=\mathbb{Z}/p\mathbb{Z}\:\:\:\textit{and}\:\:\: Ch^3(\mathcal{R}_p(G))=0\]
\end{lemme}

\begin{proof}
Since $J_p(G)=(1)$, by \cite[Example 5.3]{Eq}, the cocycle $\xi \in H^1(F,G_0)$ match with a \textit{generic} $G_0$-torsor 
in the sense of \cite{Eq}. Thus, by \cite[Proposition 3.2]{Gamma} and \cite[pp. 31, 133]{coh}, 
one has $\text{Tors}_p CH^2(\mathfrak{B})\neq 0$ 
(note that since an algebraic group of type $F_4$ or $E_8$ is simply connected,
it is of strictly inner type, and we can use material from \cite[\S 3]{Gamma}). 
The conclusion is given by \cite[Proposition 5.4]{Gamma}
\end{proof}

\medskip

Let us fix an extension $L/F$. We now prove the surjectivity of (3.4) in codimension $2$ and $3$.
By \cite[Example 4.7]{J-inv}, one has $J_p(G_L)=(0)$ or $J_p(G_L)=(1)$.

\medskip

If $J_p(G_L)=(0)$ then one has $\mathcal{R}_p(G_L)=\mathbb{Z}/p\mathbb{Z}$ by \cite[Corollary 6.7]{J-inv},
and on the other hand the motivic decomposition given in \cite[Proposition 5.18 (i)]{J-inv} implies the following 
decomposition on Chow groups for any integer $k$
\begin{equation} Ch^k(\mathcal{R}_p(G)_L)\simeq \bigoplus_{i=0}^{p-1}Ch^{k-i(p+1)}(\mathcal{R}_p(G_L)).\end{equation}
In particular, one has $Ch^k(\mathcal{R}_p(G)_L)=0$ for $k=2$ or $3$ and the conclusion follows.

\medskip

If $J_p(G_L)=(1)$ then by Lemma 3.7 one has 
$Ch^2(\mathcal{R}_p(G_L))=\mathbb{Z}/p\mathbb{Z}$ and $Ch^3(\mathcal{R}_p(G_L))=0$.
Moreover, since $J_p(G_L)=J_p(G)$, one has $\mathcal{R}_p(G_L)\simeq \mathcal{R}_p(G)_L$ 
(see \cite[Proposition 5.18 (i)]{J-inv}). Therefore, the homomorphism (3.4) is clearly surjective in codimension $3$.

We claim that it is also surjective in codimension $2$.
By (3.6) it suffices to show that the change of field 
$Ch^2(\mathfrak{B})\rightarrow Ch^2(\mathfrak{B}_L)$ is an isomorphism.
We use material and notation introduced in section 2.
Since $J_p(G)=J_p(G_L)=(1)$, the cocycles $\xi$ and $\xi_L$ match with generic $G_0$-torsors
and one consequently has $\gamma^3(\mathfrak{B})=\tau^3(\mathfrak{B})$ and
$\gamma^3(\mathfrak{B}_L)=\tau^3(\mathfrak{B}_L)$ (see \cite[Theorem 3.1(ii)]{Gamma}). It follows that
\[\gamma^{2/3}(\mathfrak{B})=\tau^{2/3}(\mathfrak{B})\:\:\: \text{and}\:\:\:
\gamma^{2/3}(\mathfrak{B}_L)=\tau^{2/3}(\mathfrak{B}_L).\]
Therefore, since $2<p+1$, the homomorphsim 
$Ch^2(\mathfrak{B})\rightarrow Ch^2(\mathfrak{B}_L)$ coincides with
\[Ch^2(\mathfrak{B})\simeq \gamma^{2/3}(\mathfrak{B})\rightarrow \gamma^{2/3}(\mathfrak{B}_L)\simeq Ch^2(\mathfrak{B}_L)\]
and the center arrow is an isomorphism by Remark 2.4.

\medskip

\medskip

\medskip

The surjectivity of (3.4) in codimension $4$ and $5$ is a direct consequence of the following statement, where $G$ is of type $E_8$ and $p=5$.
Consequently, Lemma 3.9 completes the proof of the first conclusion of Theorem 1.1 for $G$ of type $E_8$.

\begin{lemme}
\textit{For any extension $L/F$, one has}
\[Ch^4(\mathcal{R}_5(G)_L)=0 \:\:\:\textit{and}\:\:\:Ch^5(\mathcal{R}_5(G)_L)=0\]
\end{lemme}

\begin{proof}
Since $J_5(G)=(1)$, we know that $J_5(G_L)=(1)$ or $(0)$. 
If $J_5(G_L)=(0)$ then one has $R_5(G_L)=\mathbb{Z}/5\mathbb{Z}$ and the isomorphism (3.8) implies that
$Ch^4(\mathcal{R}_5(G)_L)=Ch^5(\mathcal{R}_5(G)_L)=0$. Thus, one can assume $L=F$ and we have to prove that 
$Ch^4(\mathcal{R}_5(G))=Ch^5(\mathcal{R}_5(G))=0$.

\medskip

By Proposition 2.1 there exist an extension $E/F$ and a cocycle $\xi'\in H^1(E,G_0)$ such that the topological filtration and the 
$\gamma$-filtration coincide on $K(\mathfrak{B}')$, with $\mathfrak{B}'={}_{\xi'}(G_0/B)$.
Let us denote $G'$ the variety ${}_{\xi'}G_0$. 

We claim that $J_5(G')=(1)$.
Indeed, assume that $J_5(G')=(0)$. In that case, one has $R_5(G')=\mathbb{Z}/5\mathbb{Z}$ and the isomorphism (3.6)
gives that $Ch^2(\mathfrak{B}')={\mathbb{Z}/5\mathbb{Z}}^{\oplus a_2}$.
Since $2<p+1$, it implies that $\gamma^{2/3}(\mathfrak{B}')={\mathbb{Z}/5\mathbb{Z}}^{\oplus a_2}$,
and consecutively
$\gamma^{2/3}(\mathfrak{B})={\mathbb{Z}/5\mathbb{Z}}^{\oplus a_2}$ by Remark 2.4. 
However, we have $\gamma^{2/3}(\mathfrak{B})\simeq \tau^{2/3}(\mathfrak{B})$ 
(because $\gamma^{3}(\mathfrak{B})\simeq \tau^{3}(\mathfrak{B})$ since $\xi\in H^1(F,G_0)$ is generic). 
Thus, we have $Ch^2(\mathfrak{B})={\mathbb{Z}/5\mathbb{Z}}^{\oplus a_2}$ 
which contradicts $Ch^2(\mathcal{R}_5(G))=\mathbb{Z}/5\mathbb{Z}$ and the claim is proved
(we recall that for any $i<6=p+1$, one has $\tau^{i/i+1}(X)\simeq Ch^i(X)$).  

\medskip

We now compute the groups $\gamma^{i/i+1}(\mathfrak{B}')$ for $i=3,4,5$.
Note that since $K(\mathfrak{B}')\simeq K(G_0/B)$ 
and since the description of the free group $K(G_0/B)$ in terms of generators
does not depend on the characteristic $\text{char}(E)$ of $E$ 
( see \cite[Lemma 13.3(4)]{Inv}), we can assume that $\text{char}(E)=0$ in order to compute those groups.

In that case, since $J_5(G')\neq (0)$, the isomorphism (3.6) combined with 
the following theorem (adapted from
\cite[Theorem RM.10]{OSNV} to our situation) 

\begin{thm}[Karpenko, Merkurjev]
\textit{Let $H$ be a semisimple linear algebraic group of inner type
over a field of characteristic $0$ and let $p$ be a torsion prime of $H$. 
If $J_p(H)\neq (0)$ then}
\[Ch^j(\mathcal{R}_p(H))=\left\{\begin{array}{ll}
\mathbb{Z}/p\mathbb{Z} & if\;j=0\;\;or\;j=k(p+1)-p+1,\;1\leq k \leq p-1 \\
0 & otherwise
\end{array}\right.\]
\end{thm}

\noindent gives that
\[\gamma^{i/i+1}(\mathfrak{B}')\simeq Ch^i(\mathfrak{B}')={\mathbb{Z}/5\mathbb{Z}}^{\oplus (a_{i-2}+a_i)}\:\:\:\text{for}\:i=3,4,5\]
(where the first isomorphism is due to $i<p+1$). Therefore, we get 
\[\gamma^{i/i+1}(\mathfrak{B})={\mathbb{Z}/5\mathbb{Z}}^{\oplus (a_{i-2}+a_i)}\:\:\:\text{for}\:i=3,4,5\]
(with no particular assumption on $\text{char}(F)$).
Thus, since $\tau^{3/4}(\mathfrak{B})\simeq Ch^3(\mathfrak{B})$, the isomorphism (3.6) for $k=3$ gives that 
$\tau^{3/4}(\mathfrak{B})\simeq \gamma^{3/4}(\mathfrak{B})$.
Since the $\gamma$-filtration is contained in the topological one, we get
\[\tau^{4}(\mathfrak{B})= \gamma^{4}(\mathfrak{B}),\]
which implies the existence of an exact sequence
\[0\rightarrow (\tau_5(\mathfrak{B})/\gamma_5(\mathfrak{B})) \rightarrow 
\gamma^{4/5}(\mathfrak{B}) \rightarrow \tau^{4/5}(\mathfrak{B})\rightarrow 0.\]
Thus, since $\tau^{4/5}(\mathfrak{B})\simeq Ch^4(\mathfrak{B})$, by applying the isomorphism (3.6)
for $k=4$, we get a surjection
\[{\mathbb{Z}/5\mathbb{Z}}^{\oplus (a_2+a_4)}\twoheadrightarrow Ch^4(\mathcal{R}_5(G))\oplus {\mathbb{Z}/5\mathbb{Z}}^{\oplus (a_2+a_4)},\]
which implies that $Ch^4(\mathcal{R}_5(G))=0$.

We prove that $Ch^5(\mathcal{R}_5(G))=0$ by proceeding in exactly the same way.
\end{proof}

Consequently, Proposition 3.3 is proved.
\end{proof}

Finally, we want to prove the second conclusion of Theorem 1.1 ($p=3$ if $G$ is of type $F_4$ and $p=5$ if $G$ is of type $E_8$).
First of all, since for any generic point $\zeta$ of $Y$, one has
\[1 \notin \text{deg}Ch_0(X_{F(\zeta)})\Leftrightarrow J_p(G_{F(\zeta)})=(1),\]
by Proposition 3.1 and in view of what has already been done, it is sufficient to prove the following lemma
to get the second conclusion.

\begin{lemme}
\textit{Under the assumption $J_p(G)=(1)$, one has
$Ch^{p+1}(\mathcal{R}_p(G))=0$.}
\end{lemme}

\begin{proof}
Thanks to Proposition 2.3, one can prove the lemma by proceeding in exactly the same way Lemma 3.9 has been proved.
\end{proof}

This concludes the proof of Theorem 1.1.

\bibliographystyle{acm} 
\bibliography{biblio}
\end{document}